\documentclass[12pt]{amsart}
\usepackage{amsmath,amsfonts,amssymb,amsthm,amstext,pgf,graphicx,hyperref,verbatim,textcomp,color,young,tikz,tikz-cd,cite,youngtab}
\usepackage[mathscr]{euscript}
\theoremstyle{plain}
\newtheorem{thm}{Theorem}[section]
\newtheorem{lem}[thm]{Lemma}
\newtheorem{prop}[thm]{Proposition}

\theoremstyle{definition}
\newtheorem{defn}{Definition}[section]

\newtheorem{rem}[defn]{Remark}

\DeclareMathOperator{\1}{id}
\DeclareMathOperator{\yn}{\mathcal{D}_n}
\DeclareMathOperator{\tab}{tab}
\DeclareMathOperator{\tabD}{tab_\mathcal{D}}

\DeclareMathOperator{\Tr}{Tr}

\DeclareMathOperator{\var}{Var}
\title{}
\begin{document}
\title
{Total variation cutoff for the flip-transpose top with random shuffle}
\author{Subhajit Ghosh}
\address[Subhajit Ghosh]{Department of Mathematics, Indian Institute of Science, Bangalore 560 012}
\email{gsubhajit@iisc.ac.in}
\begin{abstract}
We consider a random walk on the hyperoctahedral group $B_n$ generated by the signed permutations of the forms $(i,n)$ and $(-i,n)$ for $1\leq i\leq n$. We call this the flip-transpose top with random shuffle on $B_n$. We find the spectrum of the transition probability matrix for this shuffle. We prove that the mixing time for this shuffle is of order $n\log n$. We also show that this shuffle exhibits the cutoff phenomenon. In the appendix, we show that a similar random walk on the demihyperoctahedral group $D_n$ also has a cutoff at $\left(n-\frac{1}{2}\right)\log n$.
\end{abstract}
\maketitle
\section{Introduction}\label{intro}
Card shuffling problems are mathematically analysed by considering them as random walks on symmetric groups \cite{D2,FOW,DS,S,SchJr2,SchJr1}. In this paper, our main aim is to study the properties of a random walk on Coxeter groups of type B \cite{Coxetergrp}. This work is a generalisation of the \emph{transpose top with random shuffle} \cite{FOW,D1} to the signed permutations. A \emph{signed permutation} \cite{Coxetergrp} is a bijection $\pi$ from $\{\pm1,\dots,\pm n\}$ to itself satisfying $\pi(-i)=-\pi(i)$ for all $1\leq i\leq n$. A signed permutation is completely determined by its image on the set $[n]:=\{1,\dots,n\}$. Given a signed permutation $\pi$, we write it in window notation by $[\pi_1,\dots,\pi_n]$, where $\pi_i$ is the image of $i$ under $\pi$. The set of all signed permutations forms a group under composition and is known as the \emph{hyperoctahedral group} and is denoted by $B_n$. The subset of $B_n$ consisting of those signed permutations having even number of negative entries in their window notation form a subgroup of $B_n$, called the \emph{demihyperoctahedral group} and is denoted by $D_n$.

Suppose there are $n$ cards labelled from $1$ to $n$ and each card has two orientations namely `face up' and `face down'. Given an arrangement of these $n$ cards in a row, we associate a signed permutation $[\pi_1,\pi_2,\dots,\pi_n]$ to it in the following way: $\pi_i$ is the label of the $i\text{th}$ card (counting started from left) with sign  
\[\begin{cases}
\text{positive,}&\text{if the orientation of the card is `face up' and}\\
\text{negative,}&\text{if the orientation of the card is `face down'}.
\end{cases}\]
Thus every arrangement of the $n$ cards in a row represents a signed permutation in its window notation. We consider the following shuffle on the set of all arrangements of these $n$ cards in a row: Given an arrangement, either interchange the last card with a random card, or interchange the last card with a random card and flip both of them, with equal probability. We call this shuffle \emph{the flip-transpose top with random shuffle}. Formally, this shuffle is the random walk on $B_n$ driven by the probability measure $P$ on $B_n$, given by
\begin{equation}\label{def of P}
P(\pi)=
\begin{cases}
\frac{1}{2n},&\text{if }\pi=\1,\text{ the identity element of }B_n,\\
\frac{1}{2n},&\text{if }\pi=(i,n)\text{ for }1\leq i\leq n-1,\\
\frac{1}{2n},&\text{if }\pi=(-i,n)\text{ for }1\leq i\leq n,\\
\;0,&\text{otherwise}.
\end{cases}
\end{equation}
We study the flip-transpose top with random shuffle on $B_n$ using the representation theory of $B_n$. However, the moves are not the same for elements of the same conjugacy class (i.e., the generating measure does not take the same value at the elements of the same conjugacy class, we abbreviate such shuffle/walk as the non-conjugacy class shuffle/walk). In general, it is not easy to study non-conjugacy class walks on finite groups using the representation theory of the underlying groups. Some examples of such non-conjugacy class walks are \emph{the random-to-top} shuffle \cite{AD1,Intro_DFP} (this is an example of \emph{the Tsetlin library problem} \cite{Intro_Tse63}), \emph{the random-to-random} shuffle \cite{Intro_DSaloff-Coste1,Ayyer-Schilling-Thiery,Intro_Dieker_Saliola,Intro_R2R_Cutoff_recent}, \emph{the one-sided transposition} shuffle on the symmetric group \cite{One-sided_transposition_shuffle}, and its generalisation to the hyperoctahedral group \cite{Matheau-Raven_thesis}. In this paper, we will show that the flip-transpose top with random shuffle on $B_n$ satisfies the cutoff phenomenon and determine the mixing time for this random walk. Particularly if $||P^{*k}-U_{B_n}||_{\text{TV}}$ denotes the \emph{total variation distance} between the distribution after $k$ transitions and the \emph{stationary distribution}, then the main results of this paper are the following:
\begin{thm}\label{thm:B_n Upper Bound}
	For the flip-transpose top with random shuffle on $B_n$, we have the following:
	\begin{enumerate}
		\item $||P^{*k}-U_{B_n}||_{\emph{TV}}<\sqrt{2(e+1)}\;e^{-c}+o(1)$, for $k\geq n\log n+cn$ and $c>0$.
		\item $\lim\limits_{n\rightarrow\infty}||P^{*k_n}-U_{B_n}||_{\emph{TV}}=0$, for any $\epsilon\in (0,1)$ and $k_n=\lfloor(1+\epsilon)n\log n\rfloor.$
	\end{enumerate}
\end{thm}
\begin{thm}\label{thm:B_n lower Bound}
	For the flip-transpose top with random shuffle on $B_n$, we have the following:
	\begin{enumerate}
		\item For large $n,\;||P^{*k}-U_{B_n}||_{\emph{TV}}\geq1-\frac{2\left(3+3e^{-c}+o(1)(e^{-2c}+e^{-c}+1)\right)}{\left(1+(1+o(1))e^{-c}\right)^2}$, when $k=n\log n+cn$ and $c\ll0$.
		\item $\lim\limits_{n\rightarrow\infty}||P^{*k_n}-U_{B_n}||_{\emph{TV}}=1$, for any $\epsilon\in (0,1)$ and $k_n=\lfloor(1-\epsilon)n\log n\rfloor.$
	\end{enumerate}
\end{thm}
We will first recall some concepts and terminologies which we will use in this paper frequently.
\subsection{Representation theoretic background}
Let $V$ be a finite-dimensional complex vector space and $GL(V)$ be the group of all invertible linear operators from $V$ to itself under the composition of linear mappings. Elements of $GL(V)$ can be thought of as invertible matrices over $\mathbb{C}$. Let $G$ be a finite group, a mapping $\rho:G\rightarrow GL(V)$ is said to be a \emph{linear representation} of $G$ if $\rho(g_1g_2)=\rho(g_1)\rho(g_2)$ for all $g_1,\;g_2$ in $G$. The dimension of the vector space $V$ is said to be the \emph{dimension} of the representation $\rho$ and is denoted by $d_{\rho}$. $V$ is called the \emph{$G$-module} corresponding to the representation $\rho$ in this case. If $\mathbb{C}[G]=\{\sum_{i}c_ig_i\mid c_i\in\mathbb{C},\; g_i\in G\}$, then we define the \emph{right regular representation} $R:G\longrightarrow GL(\mathbb{C}[G])$ of $G$ by 
\[R(g)\left(\sum_{h\in G}C_hh\right)=\sum_{h\in G}C_hhg,\text{ where }C_h\in\mathbb{C}.\]
Let $H$ be a subgroup of $G$. The \emph{restriction} of the representation $\rho$ to $H$ is denoted by $\rho\downarrow^G_H$ and is defined by $\rho\downarrow^G_H(h):=\rho(h)$ for all $h\in H$. The trace of the matrix $\rho(g)$ is said to be the \emph{character} value of $\rho$ at $g$ and is denoted by $\chi^{\rho}(g)$. A vector subspace $W$ of $V$ is said to be \emph{stable} ( or `\emph{invariant}') under $\rho$ if $\rho(g)\left(W\right)\subset W$ for all $g$ in $G$. The representation $\rho$ is \emph{irreducible} if $V$ is non-trivial and $V$ has no non-trivial proper stable subspace. Two representations $(\rho_1,V_1)$ and $(\rho_2,V_2)$ of $G$ are said to be \emph{isomorphic} if there exists an invertible linear map $T:V_1\rightarrow V_2$ such that the following diagram commutes for all $g\in G$:
\[\begin{tikzcd}
V_1\arrow{r}{\rho_1(g)}\arrow[swap]{d}{T} & V_1\arrow{d}{T}\\
V_2\arrow{r}{\rho_2(g)} & V_2
\end{tikzcd}\]
If $V_1\otimes V_2$ denotes the tensor product of the vector spaces $V_1$ and $V_2$, then the tensor product of two representations $\rho_1: G\rightarrow GL(V_1)$ and $\rho_2: G\rightarrow GL(V_2)$ is a representation denoted by $(\rho_1\otimes\rho_2,V_1\otimes V_2)$ and defined by,
\[(\rho_1\otimes\rho_2)(g)(v_1\otimes v_2)=\rho_1(g)(v_1)\otimes\rho_2(g)(v_2) \text{ for }v_1\in V_1,v_2\in V_2\text{ and }g\in G.\]
We will state some results from the representation theory of finite groups without proof. For more details, see \cite{Amri,Sagan,Serre}.
\subsection{Random walks on finite groups} We first recall some terminology. Let $\;p\text{ and }q$ be two probability measures on a finite group $G$. The \emph{Fourier transform} $\widehat{p}$ of $p$ at the representation $\rho$ is defined by the matrix $\sum_{x\in G}p(x)\rho(x)$. We define the \emph{convolution $p*q$} of $p$ and $q$ by \[(p*q)(x):=\sum_{\{u,v\in G\mid uv=x\}}p(u)q(v).\] 
It can be easily seen that $\widehat{(p*q)}(\rho)=\widehat{p}(\rho)\widehat{q}(\rho)$. For the right regular representation $R$, the matrix $\widehat{p}(R)$ can be thought of as the action of the group algebra element $\sum_{g\in G}p(g)g$ on $\mathbb{C}[G]$ by multiplication on the right. 

A \emph{random walk on a finite group $G$ driven by a probability measure $p$} is a discrete time Markov chain with state space $G$ and transition probabilities $M_p(x,y)=\;p(x^{-1}y)$, $x,y\in G$. The transition matrix $M_p$ is the transpose of $\widehat{p}(R)$. If $p^{*k}$ denotes the $k$-fold convolution of $p$ with itself, then the probability of reaching state $y$ starting from state $x$ using $k$ transitions is $p^{*k}(x^{-1}y)$. The random walk is said to be \emph{irreducible} if given any two states $u$ and $v$ there exists $t$ (depending on $u$ and $v$) such that $p^{*t}(u^{-1}v)>0$. We now state the lemma regarding the irreducibility of the random walk on $G$ driven by $p$.
\begin{lem}[{\cite[Proposition 2.3]{S}}]\label{irreducibility criterion from generator}
Let $G$ be a finite group and $p$ be a probability measure on $G$. The random walk on $G$ driven by $p$ is irreducible if and only if the support of $p$ generates $G$.
\end{lem}
 A probability vector (a row vector with non-negative components which sum to one) $\Pi$ is said to be a \emph{stationary distribution} of the random walk if $\Pi$ is a left eigenvector of the transition matrix with eigenvalue $1$. There exists a unique stationary distribution for each irreducible random walk. If the random walk on $G$ driven by $p$ is irreducible, then the stationary distribution for this random walk is the uniform distribution on $G$ \cite[Section 2.2]{S}. From now on, we denote the uniform distribution on $G$ by $U_G$. Let us consider a random walk and fix one state $x\in G$. The greatest common divisor of the set of all times when it is possible for the walk to return to the starting state $x$ is said to be the \emph{period} of the state $x$. All the states of an irreducible random walk have the same period (see \cite[Lemma 1.6]{LPW}). An irreducible random walk is said to be \emph{aperiodic} if the common period for all its states is $1$.  

Let $\mu$ and $\nu$ be two probability distributions on $\Omega$. The \emph{total variation distance} between $\mu$ and $\nu$ is defined by \[||\mu-\nu||_{\text{TV}}:=\sup_{A\subset\Omega}|\mu(A)-\nu(A)|.\]
The total variation distance between two discrete distributions $\mu$ and $\nu$ is half the $\ell_1$ distance between them (see \cite[Proposition 4.2]{LPW}). If the random walk on a finite group $G$ driven by a probability measure on $G$ is irreducible and aperiodic, then the distribution after the $k\text{th}$ transition converges to the uniform measure on $G$ as $k\rightarrow\infty$. We now define the total variation cutoff phenomenon.
\begin{defn}
 Let $\{\mathcal{G}_n\}_{0}^{\infty}$ be a sequence of finite groups and $p_n$ be probability measures on $\mathcal{G}_n,\;n\geq 0$. For each $n\geq 0$, consider the irreducible and aperiodic random walks on $\mathcal{G}_n$ driven by $p_n$. We say that the \emph{total variation cutoff phenomenon} holds for the family $\{(\mathcal{G}_n,p_n)\}_0^{\infty}$ if there exists a sequence $\{\tau_n\}_0^{\infty}$ of positive real numbers such that the following hold:
\begin{enumerate}
\item $\lim\limits_{n\rightarrow\infty}\tau_n=\infty,$
\item For any $\epsilon\in (0,1)$ and $k_n=\lfloor(1+\epsilon)\tau_n\rfloor$, $\lim\limits_{n\rightarrow\infty}||p_n^{*k_{n}}-U_{\mathcal{G}_n}||_{\text{TV}}=0$ and
\item For any $\epsilon\in (0,1)$ and $k_n=\lfloor(1-\epsilon)\tau_n\rfloor$, $\lim\limits_{n\rightarrow\infty}||p_n^{*k_{n}}-U_{\mathcal{G}_n}||_{\text{TV}}=1$.
\end{enumerate}
Here $\lfloor x\rfloor$ denotes the floor  of $x$ (the largest integer less than or equal to $x$).
\end{defn}
Informally, we will say that $\{(\mathcal{G}_n,p_n)\}_0^{\infty}$ has a total variation cutoff at time $\tau_n$. Roughly the cutoff phenomenon depends on the multiplicity of the second largest eigenvalue of the transition matrix \cite{DCutoff}.
\begin{prop}\label{irreducibility and aperiodicity for type B}
The flip-transpose top with random shuffle on $B_n$ is irreducible and aperiodic.
\end{prop}
\begin{proof}
We know that the set $\{(-1,1),\;(1,2),\;(2,3),\;\dots,\;(n-1,n)\}$ generates $B_n$. Let $i$ be any integer from $[n-1]$. Then $ (i,i+1)=(i+1,n)(i,n)(i+1,n)\text{ and }(-1,1)=(1,n)(-n,n)(1,n).$ Therefore the support of the measure $P$ generates $B_n$, and hence the chain is irreducible by Lemma \ref{irreducibility criterion from generator}. Given any $\pi\in B_n$, the set of all times when it is possible for the chain to return to the starting state $\pi$ contains the integer $1$ ($\because$ the identity element of $B_n$ is in support of $P$). Therefore the period of the state $\pi$ is $1$ and hence from irreducibility all the states of this chain have period $1$. Thus this chain is aperiodic.
\end{proof}
Proposition \ref{irreducibility and aperiodicity for type B} says that the flip-transpose top with random shuffle on $B_n$ has unique stationary distribution $U_{B_n}$ and the distribution after the $k\text{th}$ transition will converge to its stationary distribution as $k\rightarrow\infty$. 

The plan of the rest of the paper is as follows: In Section \ref{sec:representation}, we will find the spectrum of the transition matrix $\widehat{P}(R)$. We will find an upper bound of $||P^{*k}-U_{B_n}||_\text{{TV}}$ for $k\geq n\log n+cn,\;c>0$ in Section \ref{sec:upper bound}. Finally, in Section \ref{sec:lower bound}, we will find a lower bound of $||P^{*k}-U_{B_n}||_\text{{TV}}$ for $k=n\log n+cn,\;c\ll 0$ (large negative number) and show that the total variation cutoff for the shuffle on $B_n$ occurs at $n\log n$.

In Appendix A, we give an outline of the irreducible representations of the demihyperoctahedral group $D_n$. We also give an idea for the deduction of irreducible representations of $D_n$ from that of $B_n$. In Appendix B, we consider a random walk on $D_n$ analogous to the flip-transpose top with random shuffle on $B_n$ and show that this random walk exhibits the total variation cutoff phenomenon with cutoff at $\left(n-\frac{1}{2}\right)\log n$.
\subsection*{Acknowledgement} I would like to thank my advisor Arvind Ayyer for proposing the problem and for all the insightful discussions during the preparation of this paper. I would like to thank the anonymous referee for the suggestion that simplifies the proof of the lower bound. I would like to acknowledge support in part by a UGC Centre for Advanced Study grant.
\section{Spectrum of The Transition Matrix $\widehat{P}(R)$}\label{sec:representation}
In this section, we find the eigenvalues of the transition matrix $\widehat{P}(R)$, the Fourier transform of $P$ at the right regular representation $R$ of $B_n$. To find the eigenvalues of $\widehat{P}(R)$, we will use the representation theory of the hyperoctahedral group $B_n$. We briefly discuss the representation theory of $B_n$. For more details, one can see \cite{MS,GK,Pushkarev}. 
\begin{defn}
A \emph{partition} $\lambda$ of a positive integer $n$ is denoted by $\lambda\vdash n$ and is defined by a finite sequence of positive integers $(\lambda_1,\lambda_2,\dots,\lambda_{\ell})$ satisfying $\lambda_1\geq\lambda_2\geq\dots\geq\lambda_{\ell}>0$ and $|\lambda|:=\sum_{i=1}^{\ell}\lambda_i=n$. The \emph{Young diagram of shape $\lambda$} is an arrangement of $n$ boxes into $\ell$ rows in a left justified way such that the $i\text{th}$ row contains $\lambda_i$ boxes for $1\leq i\leq\ell$. We use the same notation $\lambda$ to express both the partition and the Young diagram. The \emph{content} of a box in row $i$ and column $j$ of a Young diagram is the integer $j-i$. Given a Young diagram $\lambda$, its \emph{conjugate} $\lambda^{\prime}$ is obtained by reflecting $\lambda$ with respect to the diagonal consisting of boxes with content $0$. A \emph{standard Young tableau of shape $\lambda$} is a filling of the boxes of the Young diagram of shape $\lambda$ with the numbers $1,\dots,n$ such that the numbers are increasing along each row and each column. The set of all standard Young tableaux of shape $\lambda$ is denoted by $\tab(\lambda)$, and the number of standard Young tableaux of shape $\lambda$ is denoted by $d_{\lambda}$.
\end{defn}
\begin{figure}[h]
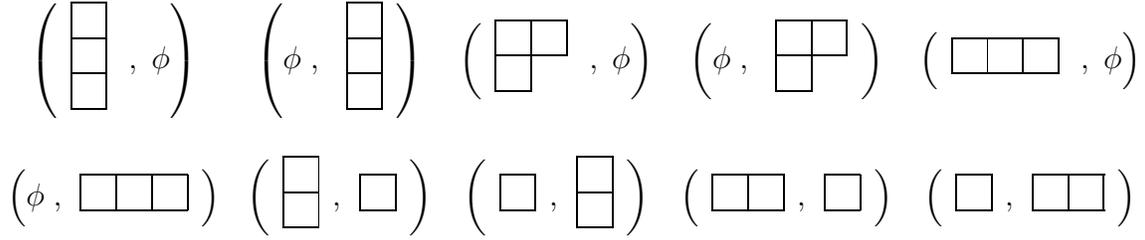

\centering
$\begin{array}{cccccc}
\left(\begin{array}{c}\yng(1,1,1)\end{array}\;,\;\phi\right)&\left(\phi\;,\;\begin{array}{c}\yng(1,1,1)\end{array}\right)&\left(\begin{array}{c}\yng(2,1)\end{array}\;,\;\phi\right)&\left(\phi\;,\;\begin{array}{c}\yng(2,1)\end{array}\right)&\left(\begin{array}{c}\yng(3)\end{array}\;,\;\phi\right)
\\\\ 
\left(\phi\;,\begin{array}{c}\yng(3)\end{array}\right)
&\left(\begin{array}{c}\yng(1,1)\end{array},\begin{array}{c}\yng(1)\end{array}\right)
&\left(\begin{array}{c}\yng(1)\end{array},\begin{array}{c}\yng(1,1)\end{array}\right)
&\left(\begin{array}{c}\yng(2)\end{array},\begin{array}{c}\yng(1)\end{array}\right)
&\left(\begin{array}{c}\yng(1)\end{array},\begin{array}{c}\yng(2)\end{array}\right)
\end{array}$
\caption{All elements of $\mathcal{D}_3$.}\label{fig: Ex of Young double-diag}
\end{figure}
\begin{defn}
Let $n$ be a positive integer. A \emph{$($Young$)$ double-diagram with $n$ boxes} $\mu$ is an (ordered) pair of Young diagrams such that the total number of boxes is $n$. We define $||\mu||=n$. The set of all double-diagrams with $n$ boxes is denoted by $\yn$. For example, the double-diagrams with $3$ boxes are listed in Figure \ref{fig: Ex of Young double-diag}. A \emph{standard $($Young$)$ double-tableau of shape $\mu$} is obtained by taking the double-diagram $\mu$ and filling its $||\mu||$ boxes (bijectively) with the numbers $1,2,\dots,||\mu||$ such that the numbers in the boxes strictly increase along each row and each column of all Young diagrams occurring in $\mu$. Let $\tabD(n,\mu)$, where $\mu\in\yn$, denote the set of all standard double-tableaux of shape $\mu$ and let $\tabD(n)=\underset{\mu\in\yn}{\cup}\tabD(n,\mu)$. For example an element of $\tabD(8)$ is given in Figure \ref{fig: Ex of content of st Young double-tab}. Let $T\in\tabD(n,\mu)$ and $i\in[n]$. Let $b_T(i)$ be the box of the Young diagram in $\mu$, in which the number $i$ resides. We denote the content of the box $b_T(i)$ by  $c(b_T(i))$. For the standard double-tableau given in Figure \ref{fig: Ex of content of st Young double-tab}, we have $c(b_T(1))=0$, $c(b_T(2))=1$, $c(b_T(3))=0$, $c(b_T(4))=1$, $c(b_T(5))=-1$, $c(b_T(6))=-1$, $c(b_T(7))=0$, $c(b_T(8))=2$.
\end{defn}
\begin{figure}[h]
\centering
$\left(\begin{array}{c}\young(348,67)\end{array},\begin{array}{c}\young(12,5)\end{array}\right)$
\caption{An element of $\tabD(8)$.}\label{fig: Ex of content of st Young double-tab}
\end{figure}
\begin{defn}\label{def: YJM for B_n}
The \emph{Young-Jucys-Murphy} elements $X_1,X_2,\dots,X_n$ of $\mathbb{C}[B_n]$ are defined by $X_1= 0$ and $X_i= \displaystyle\sum_{k=1}^{i-1}(k,i)+\displaystyle\sum_{k=1}^{i-1}(-k,i)$, for all $2\leq i\leq n$.
\end{defn}
\begin{defn}\label{def: GZ-}
Let $\mu\in\widehat{B}_n$ (set of all irreducible representations of $B_n$) and consider the $B_n$-module $V^{\mu}$. Since the branching is simple \cite[Section 3]{MS}, the decomposition into irreducible $B_{n-1}$-modules is canonical and is given by 
\[V^{\mu}=\underset{\lambda}{\oplus}V^{\lambda},\]
where the sum is over all $\lambda\in \widehat{B}_{n-1}$, with $\lambda\nearrow\mu$ (i.e. there is an edge from $\lambda$ to $\mu$ in the branching multi-graph). Iterating this decomposition of $V^{\mu}$ into irreducible $B_1$-modules, we obtain
\begin{equation}\label{eq:GZ-}
V^{\mu}=\underset{T}{\oplus}v_{T},
\end{equation}
where the sum is over all possible chains $T=\mu_1\nearrow\mu_2\nearrow\dots\nearrow\mu_n$ with $\mu_i\in \widehat{B}_i$ and $\mu_n=\mu$. We call \eqref{eq:GZ-} the \emph{Gelfand-Tsetlin} decomposition of $V^{\mu}$ and each $v_T$ in \eqref{eq:GZ-} a \emph{Gelfand-Tsetlin} vector of $V^{\mu}$. We note that if $0\neq v_T,\text{ then }\mathbb{C}[B_i]v_T=V^{\mu_i}$. The Gelfand-Tsetlin vectors of $V^{\mu}$ form a basis of $V^{\mu}$.
\end{defn}
The irreducible representations of $B_n$ are parametrised by elements of $\yn$ \cite[Lemma 6.2, Theorem 6.4]{MS}. We may index the Gelfand-Tsetlin vectors of $V^{\mu}$ by standard double-tableaux  of shape $\mu$ for $\mu\in\yn$ \cite[Theorem 6.5]{MS} and write the Gelfand-Tsetlin decomposition as 
\[V^{\mu}=\underset{T\in\tabD(n,\mu)}{\oplus}v_T.\]
Let $\mu=\left(\mu^{(1)},\mu^{(2)}\right)\in\yn$ and $T\in\tabD(n,\mu)$. Then the action \cite[Theorem 6.5]{MS} of the Young-Jucys-Murphy elements $X_i$ and the signed permutation $(i,-i)$ on $v_T$ are given by 
\begin{equation}\label{action on GZ-basis}
 \begin{split}
 X_i\; v_T &= 2c(b_T(i))\;v_T\text{ for all }i\in [n],\\
 (-i,i)\; v_T &=\begin{cases}
 v_T&\text{if }b_T(i)\text{ is in }\mu^{(1)}\\
 -v_T&\text{if }b_T(i)\text{ is in }\mu^{(2)}
 \end{cases} \text{ for all }i\in [n].
 \end{split}
\end{equation}
\begin{rem}
		The components of the elements of $\yn$ are indexed by the irreducible representations of the cyclic group of order two. In this paper, we adopt the convention that the first component of the elements of $\yn$ is indexed by the trivial representation.
\end{rem}
We now come to our main problem of finding the eigenvalues of the transition matrix $\widehat{P}(R)$. The eigenvalues of $\widehat{P}(R)$ are the eigenvalues of $\frac{1}{2n}\left(\1+(-n,n)+X_n\right)$ acting on $\mathbb{C}[B_n]$ by multiplication on the right.
The following theorem gives the eigenvalues of $\widehat{P}(R)$.
	\begin{thm}\label{thm: evalues for B}
	For each $\mu=\left(\mu^{(1)},\mu^{(2)}\right)\in\yn$ satisfying $m:=|\mu^{(1)}|\in\{0,1,\dots,\lfloor\frac{n}{2}\rfloor\}$, let $T\in\tabD(n,\mu)$. Then $\frac{c(b_T(n))+1}{n}$ and $\frac{c(b_T(n))}{n}$ are eigenvalues of $\widehat{P}(R)$ with multiplicity $M(\mu)$ each, where 
	\begin{equation}\label{multiplicity of evalue}
	M(\mu)=\begin{cases}
	\;\;{n\choose m}d_{\mu^{(1)}}d_{\mu^{(2)}},&\text{if }0\leq m < \frac{n}{2},\\
	\frac{1}{2}{n\choose m}d_{\mu^{(1)}}d_{\mu^{(2)}},&\text{if }m= \frac{n}{2}\;(\text{when $n$ is even}).
	\end{cases}
	\end{equation}
\end{thm}
\begin{proof}
For each $\mu=\left(\mu^{(1)},\mu^{(2)}\right)\in\yn$, we have another double-diagram $\tilde{\mu}$ with $n$ boxes such that $\tilde{\mu}=\left(\mu^{(2)},\mu^{(1)}\right)$. We first find the eigenvalues of the matrix $\widehat{P}(R)$ in the irreducible $B_n$-modules $V^{\mu}$ and $V^{\tilde{\mu}}$. For each $T=\left(T_1,T_2\right)\in\tabD(n,\mu)$, $\widetilde{T}=\left(T_2,T_1\right)\in\tabD(n,\tilde{\mu})$. If $b_T(n)$ is in $\mu^{(1)}$, then $b_{\widetilde{T}}(n)$ is in $\tilde{\mu}^{(2)}$. Without loss of generality, let us assume that $b_T(n)$ is in $\mu^{(1)}$ and $b_{\widetilde{T}}(n)$ is in $\tilde{\mu}^{(2)}$. Let us recall $v_T$ (respectively $v_{\widetilde{T}}$) is the Gelfand-Tsetlin vector of $V^{\mu}$ (respectively $V^{\tilde{\mu}}$).
From \eqref{action on GZ-basis} we have $(-n,n)\;v_T= v_T \text{ and }X_n\;v_T= 2c(b_T(n))\;v_T$, which implies the following:
\begin{align}\label{eq: evalues1}
\left(\1+(-n,n)+X_n\right)\;v_T=&\left(1+1+2c(b_T(n))\right)\;v_T\nonumber\\ =&(2c(b_T(n))+2)\;v_T.
\end{align}
Since $\{v_T: {T\in\tabD(n,\mu)}\}$ form a basis of $V^{\mu}$, the eigenvalues of the action of $\left(\1+(-n,n)+X_n\right)$ on $V^{\mu}$ can be obtained from \eqref{eq: evalues1}.
Now using \eqref{action on GZ-basis} again we have $(-n,n)\;v_{\widetilde{T}}=-v_{\widetilde{T}}\text{ and }X_n\;v_{\widetilde{T}}= 2c(b_{\widetilde{T}}(n))\;v_{\widetilde{T}}$, thus
\begin{align}\label{eq: evalues2}
\left(\1+(-n,n)+X_n\right)\;v_{\widetilde{T}}=&\left(1-1+2c(b_T(n))\right)\;v_{\widetilde{T}}\nonumber\\
= & 2c(b_T(n))\;v_{\widetilde{T}}.
\end{align}
 Therefore the eigenvalues of the action of $\left(\1+(-n,n)+X_n\right)$ on $V^{\tilde{\mu}}$ are obtained from \eqref{eq: evalues2}, as $\{v_{\widetilde{T}}: {{\widetilde{T}}\in\tabD(n,\tilde{\mu})}\}$ form a basis of $V^{\tilde{\mu}}$.  
Thus considering the action of $\frac{1}{2n}\left(\1+(-n,n)+X_n\right)$ on $V^{\mu}$ and $V^{\tilde{\mu}}$ simultaneously, the eigenvalues of $\widehat{P}(R)$ are given by $\frac{c(b_T(n))+1}{n}\text{ and }\frac{c(b_T(n))}{n}$ for each $T\in\tabD(n,\mu)$. 

Now we know that the multiplicity of every irreducible representation in the right regular representation is equal to its dimension. Therefore the multiplicity of the eigenvalues are dim$(V^{\mu})={n\choose m}d_{\mu^{(1)}}d_{\mu^{(2)}}=\text{dim}(V^{\tilde{\mu}})$ if $0\leq m < \frac{n}{2}$ and the multiplicity of the eigenvalues are $\frac{1}{2}{n\choose m}d_{\mu^{(1)}}d_{\mu^{(2)}}$ if $m=\frac{n}{2}$ (when $n$ is even). The multiplicity of the eigenvalues for the case of $m=\frac{n}{2}$ is half of the dimension of the corresponding $B_n$-module because of the following: In this case $m=n-m$. Thus both $\mu=(\mu^{(1)},\mu^{(2)})$ and $\tilde{\mu}=(\mu^{(2)},\mu^{(1)})$ are in $\yn$ such that their first component is a partition of $m$ and the second component is a partition of $n-m$. Therefore while computing the eigenvalues of $\widehat{P}(R)$ by considering the irreducible $B_n$-modules $V^{\mu}$ and $V^{\tilde{\mu}}$, each space is counted twice. Now the proof of the theorem follows from the fact that all the irreducible representations of $B_n$ are parameterised by $\yn$.
\end{proof}
\section{Upper bound of total variation distance}\label{sec:upper bound}
In this section, we will prove the theorem giving an upper bound of the total variation distance $||P^{*k}-U_{B_n}||_\text{{TV}}$ for $k\geq n\log n+cn,\;c>0$. Given a positive integer $\ell$, throughout this section we write $\lambda\vdash\ell$ to denote $\lambda$ is a partition of $\ell$. Let us recall that $\tab(\lambda)$ denote the set of all standard Young tableaux of shape $\lambda$.
\begin{lem}[{Upper bound lemma, \cite[Lemma 4.2]{D1}}]\label{Upper Bound Lemma}
Let $p$ be a probability measure on a finite group  $G$ such that $p(x)=p(x^{-1})$ for all $x\in G$. Suppose the random walk on $G$ driven by $p$ is irreducible. Then we have the following
\[||p^{*k}-U_{G}||^2_{\emph{TV}}\leq\frac{1}{4}\displaystyle\sum_{\rho}^{*}d_{\rho}\Tr\left(\left(\widehat{p}(\rho)\right)^{2k}\right),\]
where the sum is over all non-trivial irreducible representations  $D_{\rho}$ of $G$ and $d_{\rho}$ is the dimension of $D_{\rho}$.
\end{lem}
\begin{lem}\label{lem: upper bound 1}
Let $m$ be any positive integer satisfying $1\leq m\leq\frac{n}{2}$ and $\mu=\left(\mu^{(1)},\;\mu^{(2)}\right)\in\yn$ be such that $|\mu^{(1)}|=m,\;|\mu^{(2)}|=n-m$. If $\mu^{(i)}_1\;($respectively $\mu^{(i)'}_1)$ denotes the largest part of the partition $\mu^{(i)}\;($respectively its conjugate $\mu^{(i)'})$  for $i=1,2$, then
\begin{equation*}
\displaystyle\sum_{T\in\tabD(n,\mu)}\left(\frac{c(b_T(n))+x}{n}\right)^{2k}<{n\choose m}d_{\mu^{(2)}}d_{\mu^{(1)}}\displaystyle\sum_{i=1}^{2}\left(\left(\frac{\mu^{(i)}_1}{n}\right)^{2k}+\left(\frac{\mu^{(i)'}_1}{n}\right)^{2k}\right),
\end{equation*}
with $x=0,1$.
\end{lem}
\begin{proof}
The set $\tabD(n,\mu)$ is a disjoint union of the sets $\mathcal{T}_1=\{(T_1,T_2)\in\tabD(n,\mu):b_T(n)\text{ is in }T_1\}$ and $\mathcal{T}_2=\{(T_1,T_2)\in\tabD(n,\mu):b_T(n)\text{ is in }T_2\}$. 
Therefore we have
\begin{equation}\label{lem: UB1}
\displaystyle\sum_{T\in\tabD(n,\mu)}\left(\frac{c(b_T(n))+x}{n}\right)^{2k}=\displaystyle\sum_{T\in\mathcal{T}_1}\left(\frac{c(b_T(n))+x}{n}\right)^{2k}+\displaystyle\sum_{T\in\mathcal{T}_2}\left(\frac{c(b_T(n))+x}{n}\right)^{2k}.
\end{equation} 
Now the right hand side of \eqref{lem: UB1} is equal to
\begin{align}
&\hspace*{-1cm}{n-1\choose n-m}d_{\mu^{(2)}}\displaystyle\sum_{T_1\in\tab(\mu^{(1)})}\left(\frac{c(b_{T_1}(m))+x}{n}\right)^{2k}+{n-1\choose m}d_{\mu^{(1)}}\displaystyle\sum_{T_2\in\tab(\mu^{(2)})}\left(\frac{c(b_{T_2}(n-m))+x}{n}\right)^{2k}\nonumber\\
&<{n\choose m}\left(d_{\mu^{(2)}}\displaystyle\sum_{T_1\in\tab(\mu^{(1)})}\left(\frac{c(b_{T_1}(m))+x}{n}\right)^{2k}+d_{\mu^{(1)}}\displaystyle\sum_{T_2\in\tab(\mu^{(2)})}\left(\frac{c(b_{T_2}(n-m))+x}{n}\right)^{2k}\right)\nonumber\\
&\leq{n\choose m}d_{\mu^{(2)}}d_{\mu^{(1)}}\left(\left(\frac{\mu^{(1)}_1}{n}\right)^{2k}+\left(\frac{\mu^{(1)'}_1}{n}\right)^{2k}+\left(\frac{\mu^{(2)}_1}{n}\right)^{2k}+\left(\frac{\mu^{(2)'}_1}{n}\right)^{2k}\right).\label{lem: UB1.0}
\end{align}
The inequality in \eqref{lem: UB1.0} follows from the fact: If $\lambda_1$ (respectively $\lambda_1^{\prime}$) denotes the largest part of the partition $\lambda$ (respectively its conjugate $\lambda^{\prime}$), then for all $T\in\tab(\lambda)$ and for $x=0,1$ we have
\[\left(\frac{c(b_T(|\lambda|))+x}{n}\right)^{2k}\leq\max\bigg\{\left(\frac{\lambda_1-1+x}{n}\right)^{2k},\left(\frac{\lambda_1^{\prime}-1-x}{n}\right)^{2k}\bigg\}<\left(\frac{\lambda_1}{n}\right)^{2k}+\left(\frac{\lambda_1^{\prime}}{n}\right)^{2k}.\qedhere\]
\end{proof}
\begin{lem}\label{lem: upper bound 2}
Let $\ell$ be a positive integer. For a partition $\lambda$ of $\ell$, if $\lambda_1$ denotes the largest part of $\lambda$, then
\[\displaystyle\sum_{\lambda\vdash\ell}d_{\lambda}^2\left(\frac{\lambda_1}{\ell}\right)^{2k}<e^{\ell^2e^{-\frac{2k}{\ell}}}.\]
\end{lem}
\begin{proof}
For any partition $\zeta$ of $\ell-\lambda_1$ with largest part $\zeta_1$ less than or equal to $\lambda_1$, we have $d_{\lambda}\leq{\ell \choose\lambda_1}d_{\zeta}$. Therefore $\displaystyle\sum_{\lambda\vdash\ell}d_{\lambda}^2\left(\frac{\lambda_1}{\ell}\right)^{2k}$ is less than or equal to
\begin{align}\label{eq: upper bound 2}
\displaystyle\sum_{\lambda_1=1}^{\ell}\;\displaystyle\sum_{\substack{\zeta\vdash(\ell-\lambda_1)\\\zeta_1\leq\lambda_1}}{\ell \choose\lambda_1}^2d_{\zeta}^2\left(\frac{\lambda_1}{\ell}\right)^{2k}
&\leq\displaystyle\sum_{\lambda_1=1}^{\ell}{\ell \choose\lambda_1}^2\left(\frac{\lambda_1}{\ell}\right)^{2k}\displaystyle\sum_{\zeta\vdash(\ell-\lambda_1)}d_{\zeta}^2\nonumber\\
&=\displaystyle\sum_{\lambda_1=1}^{\ell}{\ell \choose\ell-\lambda_1}^2(\ell-\lambda_1)!\left(1-\frac{\ell-\lambda_1}{\ell}\right)^{2k}.
\end{align}
Now writing $t=\ell-\lambda_1$ and using $1-x\leq e^{-x}$ for $x\geq 0$, the expression in \eqref{eq: upper bound 2} less than or equal to $\displaystyle\sum_{t=0}^{\ell-1}{\ell \choose t}^2t!e^{-\frac{2kt}{\ell}}$. Thus we have
\begin{align*}
\displaystyle\sum_{\lambda\vdash\ell}d_{\lambda}^2\left(\frac{\lambda_1}{\ell}\right)^{2k}\leq\displaystyle\sum_{t=0}^{\ell-1}{\ell \choose t}^2t!e^{-\frac{2kt}{\ell}}&=\displaystyle\sum_{t=0}^{\ell-1}\frac{\left(\ell(\ell-1)\dots(\ell-t+1)\right)^2}{t!}e^{-\frac{2kt}{\ell}}\\
&\leq \displaystyle\sum_{t=0}^{\ell-1}\frac{\left(\ell^2e^{-\frac{2k}{\ell}}\right)^{t}}{t!}<e^{\ell^2e^{-\frac{2k}{\ell}}}.\qedhere
\end{align*}
\end{proof}
\begin{proof}[Proof of Theorem \ref{thm:B_n Upper Bound}]
We know that the trace of the $(2k)\text{th}$ power of a matrix is the sum of the $(2k)\text{th}$ powers of its eigenvalues. Therefore Lemma \ref{Upper Bound Lemma} implies $4||P^{*k}-U_{B_n}||^2_{\text{TV}}$ is bounded above by the sum of $(2k)\text{th}$ powers of the non-largest eigenvalues (which are strictly less the largest eigenvalue $1$) of $\widehat{P}(R)$.
Thus from Theorem \ref{thm: evalues for B}, we have
\begin{align}\label{eq: B_n UB1}
 4||P^{*k}-U_{B_n}||^2_{\text{TV}}&\leq\left(\frac{n-1}{n}\right)^{2k}+\displaystyle\sum_{\substack{\lambda\vdash n\\\lambda\neq(n)}}d_{\lambda}\left(\displaystyle\sum_{T\in\tab(\lambda)}\left(\left(\frac{c(b_T(n))+1}{n}\right)^{2k}+\left(\frac{c(b_T(n))}{n}\right)^{2k}\right)\right)\\
+&\displaystyle\sum_{m=1}^{\lfloor\frac{n}{2}\rfloor}\displaystyle\sum_{\substack{\mu^{(1)}\vdash m\\\mu^{(2)}\vdash (n-m)\\\mu=\left(\mu^{(1)},\;\mu^{(2)}\right)}}M(\mu)\left(\displaystyle\sum_{T\in\tabD(n,\mu)}\left(\left(\frac{c(b_T(n))+1}{n}\right)^{2k}+\left(\frac{c(b_T(n))}{n}\right)^{2k}\right)\right).\nonumber
\end{align}
$M(\mu)$ is defined in \eqref{multiplicity of evalue} and can be written as $M(\mu)=\mathbb{I}(n,m){n\choose m}d_{\mu^{(1)}}d_{\mu^{(2)}}$, where
\begin{equation*}
\mathbb{I}(n,m)=\begin{cases}
1&\text{if }0\leq m < \frac{n}{2},\\
\frac{1}{2}&\text{if }m= \frac{n}{2}\;(\text{when $n$ is even}).
\end{cases}
\end{equation*}
 Using Lemma \ref{lem: upper bound 1}, the third term in the right hand side of \eqref{eq: B_n UB1} is less than the following expression
\begin{align}
&\displaystyle\sum_{m=1}^{\lfloor\frac{n}{2}\rfloor}\displaystyle\sum_{\substack{\mu^{(1)}\vdash m\\\mu^{(2)}\vdash (n-m)\\\mu=\left(\mu^{(1)},\;\mu^{(2)}\right)}}2M(\mu){n\choose m}d_{\mu^{(2)}}d_{\mu^{(1)}}\sum_{i=1}^{2}\left(\left(\frac{\mu^{(i)}_1}{n}\right)^{2k}+\left(\frac{\mu^{(i)'}_1}{n}\right)^{2k}\right)\nonumber\\
&=2\displaystyle\sum_{m=1}^{\lfloor\frac{n}{2}\rfloor}\sum_{\substack{\mu^{(1)}\vdash m\\\mu^{(2)}\vdash (n-m)\\\mu=\left(\mu^{(1)},\;\mu^{(2)}\right)}}2M(\mu){n\choose m}d_{\mu^{(2)}}d_{\mu^{(1)}}\left(\left(\frac{\mu^{(1)}_1}{n}\right)^{2k}+\left(\frac{\mu^{(2)}_1}{n}\right)^{2k}\right)\label{eq: B_n UB1.1-2.0}\\
&=4\displaystyle\sum_{m=1}^{\lfloor\frac{n}{2}\rfloor}\mathbb{I}(n,m){n\choose m}^2\displaystyle\sum_{\substack{\mu^{(1)}\vdash m\\\mu^{(2)}\vdash (n-m)}}d_{\mu^{(1)}}^2d_{\mu^{(2)}}^2\left(\left(\frac{\mu^{(1)}_1}{n}\right)^{2k}+\left(\frac{\mu^{(2)}_1}{n}\right)^{2k}\right)\nonumber\\
&=4\displaystyle\sum_{m=1}^{\lfloor\frac{n}{2}\rfloor}\mathbb{I}(n,m){n\choose m}^2\left((n-m)!\displaystyle\sum_{\mu^{(1)}\vdash m}d_{\mu^{(1)}}^2\left(\frac{\mu^{(1)}_1}{n}\right)^{2k}+m!\displaystyle\sum_{\mu^{(2)}\vdash (n-m)}d_{\mu^{(2)}}^2\left(\frac{\mu^{(2)}_1}{n}\right)^{2k}\right).\label{eq: B_n UB1.1-2}
\end{align}
The equality in \eqref{eq: B_n UB1.1-2.0} holds because
	\[\displaystyle\sum_{\substack{\mu^{(1)}\vdash m\\\mu^{(2)}\vdash (n-m)\\\mu=\left(\mu^{(1)},\;\mu^{(2)}\right)}}2M(\mu){n\choose m}d_{\mu^{(2)}}d_{\mu^{(1)}}\left(\frac{\mu^{(i)'}_1}{n}\right)^{2k}=\displaystyle\sum_{\substack{\mu^{(1)}\vdash m\\\mu^{(2)}\vdash (n-m)\\\mu=\left(\mu^{(1)},\;\mu^{(2)}\right)}}2M(\mu){n\choose m}d_{\mu^{(2)}}d_{\mu^{(1)}}\left(\frac{\mu^{(i)}_1}{n}\right)^{2k}\]
for $i=1,2$. Now the definition of $\mathbb{I}(n,m)$ and  \[\displaystyle\sum_{m=1}^{\lfloor\frac{n}{2}\rfloor}\mathbb{I}(n,m){n\choose m}^2 m!\hspace*{-1.5ex}\displaystyle\sum_{\mu^{(2)}\vdash (n-m)}\hspace*{-1.5ex}d_{\mu^{(2)}}^2\left(\frac{\mu^{(2)}_1}{n}\right)^{2k}\hspace*{-1ex}=\hspace*{-0.75ex}\displaystyle\sum_{t=\lceil\frac{n}{2}\rceil}^{n-1}\mathbb{I}(n,n-t){n\choose n-t}^2 (n-t)!\displaystyle\sum_{\mu^{(2)}\vdash t}d_{\mu^{(2)}}^2\left(\frac{\mu^{(2)}_1}{n}\right)^{2k}\]
implies that the expression \eqref{eq: B_n UB1.1-2} is equal to
\begin{equation}\label{eq: B_n UB1.2}
4\displaystyle\sum_{m=1}^{n-1}{n\choose m}^2(n-m)!\displaystyle\sum_{\mu^{(1)}\vdash m}d_{\mu^{(1)}}^2\left(\frac{\mu^{(1)}_1}{n}\right)^{2k}.
\end{equation}
Replacing $\ell$ (respectively $\lambda$) by $m$ (respectively $\mu^{(1)}$) in Lemma \ref{lem: upper bound 2}, we have $\displaystyle\sum_{\mu^{(1)}\vdash m}d_{\mu^{(1)}}^2\left(\frac{\mu^{(1)}_1}{m}\right)^{2k}<e^{m^2e^{-\frac{2k}{m}}}$. Thus $\displaystyle\sum_{\mu^{(1)}\vdash m}d_{\mu^{(1)}}^2\left(\frac{\mu^{(1)}_1}{m}\right)^{2k}< e$, if $k\geq m\log m$. Therefore when $k\geq n\log n$ (which implies $k\geq m\log m$), the expression in \eqref{eq: B_n UB1.2} and hence the third term in the right hand side of \eqref{eq: B_n UB1} is less than
\begin{align}\label{eq: B_n UB1.3}
4e\displaystyle\sum_{m=1}^{n-1}{n\choose m}^2(n-m)!\left(\frac{m}{n}\right)^{2k}= 4e\displaystyle\sum_{t=1}^{n-1}{n\choose t}^2t!\left(1-\frac{t}{n}\right)^{2k}&< 4e\displaystyle\sum_{t=1}^{n-1}\frac{\left(n^2e^{-\frac{2k}{n}}\right)^{t}}{t!}\nonumber\\
&<4e\left(e^{n^2e^{-\frac{2k}{n}}}-1\right).
\end{align}
Now we consider the second term in the right hand side of \eqref{eq: B_n UB1}. The second term in the right hand side of \eqref{eq: B_n UB1} is bounded above by
\begin{equation}\label{eq: B_n UB2.1}
2\sum_{\substack{\lambda\vdash n\\\lambda\neq (n),(1^n)}}d_{\lambda}^2\left(\left(\frac{\lambda_1}{n}\right)^{2k}+\left(\frac{\lambda_1^{\prime}}{n}\right)^{2k}\right)+\left(\frac{n-2}{n}\right)^{2k}+\left(\frac{n-1}{n}\right)^{2k}.
\end{equation}
Now using $\displaystyle\sum_{\substack{\lambda\vdash n\\\lambda\neq (n),(1^n)}}d_{\lambda}^2\left(\frac{\lambda_1^{\prime}}{n}\right)^{2k}=\displaystyle\sum_{\substack{\lambda\vdash n\\\lambda\neq (n),(1^n)}}d_{\lambda}^2\left(\frac{\lambda_1}{n}\right)^{2k}$, the expression in \eqref{eq: B_n UB2.1} is equal to
\begin{align}\label{eq: B_n UB2.2}
4&\sum_{\substack{\lambda\vdash n\\\lambda\neq (n),(1^n)}}d_{\lambda}^2\left(\frac{\lambda_1}{n}\right)^{2k}+\left(1-\frac{2}{n}\right)^{2k}+\left(1-\frac{1}{n}\right)^{2k}\nonumber\\
<&\;4\left(\sum_{\lambda\vdash n}d_{\lambda}^2\left(\frac{\lambda_1}{n}\right)^{2k}-1\right)+e^{-\frac{4k}{n}}+e^{-\frac{2k}{n}}.
\end{align}
The right hand side of the expression \eqref{eq: B_n UB2.2} and hence the second term in the right hand side of \eqref{eq: B_n UB1} is less than $4\left(e^{n^2e^{-\frac{2k}{n}}}-1\right)+e^{-\frac{4k}{n}}+e^{-\frac{2k}{n}}$ by Lemma \ref{lem: upper bound 2}. Thus the inequality \eqref{eq: B_n UB1} becomes
\begin{equation}\label{eq: B_n UB}
4||P^{*k}-U_{B_n}||^2_{\text{TV}}\leq 2e^{-\frac{2k}{n}}+(4+4e)\left(e^{n^2e^{-\frac{2k}{n}}}-1\right)+e^{-\frac{4k}{n}},\quad\text{for }k\geq n\log n.
\end{equation}
Now if $k\geq n\log n+cn$ and $c>0$, then the right hand side of \eqref{eq: B_n UB} becomes
\[(4e+4)\left(e^{e^{-2c}}-1\right)+\frac{2e^{-2c}}{n^2}+\frac{e^{-4c}}{n^4}<(8e+8)e^{-2c}+o(1).\]
This proves the first part of the theorem. Now for $\epsilon\in(0,1),\;k_n=\lfloor(1+\epsilon)n \log n\rfloor$ implies, $k_n\geq(1+\epsilon)n \log n$. Thus the right hand side of \eqref{eq: B_n UB} is bounded above by 
\[(4e+4)\left(e^{\frac{1}{n^{2\epsilon}}}-1\right)+2n^{-2(1+\epsilon)}+n^{-4(1+\epsilon)}.\] 
Therefore the proof of the second part follows from
\[\lim\limits_{n\rightarrow\infty}(4e+4)\left(e^{\frac{1}{n^{2\epsilon}}}-1\right)+\frac{2}{n^{2(1+\epsilon)}}+\frac{1}{n^{4(1+\epsilon)}}=0.\qedhere\]
\end{proof}
\section{Lower bound of total variation distance}\label{sec:lower bound}
In this section, we will find a lower bound of the total variation distance $||P^{*k}-U_{B_n}||_\text{{TV}}$ for $k=\log n+cn,\;c\ll0$. We define a group homomorphism from $B_n$ onto the symmetric group $S_n$ which projects the flip-transpose top with random shuffle on $B_n$ to the transpose top with random shuffle on $S_n$. We begin with a variant of the transpose top with random shuffle on $S_n$. This will be useful in obtaining the lower bound of $||P^{*k}-U_{B_n}||_\text{{TV}}$.

Given $0<a<1$, we define a probability measure $\mathscr{P}_a$ on the symmetric group $S_n$ as follows:
\begin{equation}\label{eq:defn_of_script_P_a}
\mathscr{P}_a(\pi)=\begin{cases}
a&\text{ if }\pi=s_{\1},\text{ the identity element of }S_n,\\
\frac{1-a}{n-1}&\text{ if }\pi=s_{(i,n)},\;1\leq i<n,
\end{cases}\quad\quad\text{ for }\pi\in S_n,
\end{equation}
where $s_{(i,n)}$ denotes the transposition in $S_n$ interchanging $i$ and $n$. We prove a theorem which provides a lower bound of $||\mathscr{P}_a^{*k}-U_{S_n}||_{\text{TV}}$. Although the proof is straightforward and uses techniques \cite[Chapter 5(C), p.=27]{D1} from the transpose top with random shuffle on $S_n$, we prove it to make this paper self contained. Recall that $\widehat{\mathscr{P}}_a(R)$ denotes the Fourier transform of $\mathscr{P}_a$ at the right regular representation $R$ of $S_n$. Thus $\widehat{\mathscr{P}}_a(R)$ is the transition matrix for the random walk on $S_n$ driven by $\mathscr{P}_a$ (since $\widehat{\mathscr{P}}_a(R)$ is symmetric). If we denote the $n$th Young-Jucys-Murphy element of $\mathbb{C}[S_n]$ using notation $X_n'$, then $\widehat{\mathscr{P}}_a(R)$ is the action of $as_{\1}+\frac{1-a}{n-1}X_n'$ on $\mathbb{C}[S_n]$ by multiplication on the right. Moreover given an irreducible $S_n$-module (Specht module) $S^{\lambda}$ indexed by $\lambda\vdash n$, the actions of $X_n'$ on the Gelfand-Tsetlin basis vectors of $S^{\lambda}$ are given as follows: $X_n'(u_T)=c(b_T(n))u_T$, where $u_T$ denotes the Gelfand-Tsetlin basis vector of $S^{\lambda}$ indexed by $T\in\tab(\lambda)$ and recall that $c(b_T(n))$ is the content of the box containing $n$ in $T$ \cite{VO,CST}. Let us define a random variable $\mathfrak{f}$ on $S_n$ given as follows:
\[\mathfrak{f}(\pi):=\text{numbers of fixed points of }\pi,\;\text{ for }\pi\in S_n.\]
The expected value of $\mathfrak{f}$ with respect to the uniform distribution $U_{S_n}$ is given by $E_{U}\left(\mathfrak{f}\right)=1$ \cite[p.=27, eq. (5.12)]{D1}. Moreover if $\widehat{\mathscr{P}}_a(R)\big|_{\lambda}$ denotes the restriction of $\widehat{\mathscr{P}}_a(R)$ to the irreducible $S_n$-module $S^{\lambda}$, then the expected value of $\mathfrak{f}$ with respect to the distribution $\mathscr{P}_a^{*k}$ is given by
\begin{align}\label{eq:E_k(.)}
E_{a,k}\left(\mathfrak{f}\right)&=\Tr\left(\left(\widehat{\mathscr{P}}_a(R)\big|_{(n)}\right)^k\right)+\Tr\left(\left(\widehat{\mathscr{P}}_a(R)\big|_{(n-1,1)}\right)^k\right),\text{ \cite[p.=28, eq.(5.13)]{D1}}\nonumber\\
&=1+(n-2)\left(1-\frac{1-a}{n-1}\right)^{k}+\left(\frac{an-1}{n-1}\right)^k,\;\text{ see Table \ref{Table}}\nonumber\\
&\approx1+(n-2)e^{-\left(\frac{1-a}{n-1}\right)k}+\left(\frac{an-1}{n-1}\right)^k.
\end{align}
Here ` $\approx$' means `asymptotic to' i.e. $a_n\approx b_n$ means $\lim\limits_{n\rightarrow\infty}\frac{a_n}{b_n}=1$.
The expectation of $\mathfrak{f}^2$ with respect to the distribution $\mathscr{P}_a^{*k}$ is given by
\begin{align}\label{eq:E_k(.^2)}
	E_{a,k}\left(\mathfrak{f}^2\right)&=2\Tr\left(\left(\widehat{\mathscr{P}}_a(R)\big|_{(n)}\right)^k\right)+3\Tr\left(\left(\widehat{\mathscr{P}}_a(R)\big|_{(n-1,1)}\right)^k\right)\nonumber\\
	&\quad+\Tr\left(\left(\widehat{\mathscr{P}}_a(R)\big|_{(n-2,2)}\right)^k\right)+\Tr\left(\left(\widehat{\mathscr{P}}_a(R)\big|_{(n-2,1,1)}\right)^k\right),\text{ \cite[p.=28, eq.(5.14)]{D1}}\nonumber\\
	&=2+3\left((n-2)\left(1-\frac{1-a}{n-1}\right)^{k}+\left(\frac{an-1}{n-1}\right)^k\right)\nonumber\\
	&\quad+\left(\frac{(n-1)(n-4)}{2}\left(1-\frac{2(1-a)}{n-1}\right)^k+(n-2)a^k\right)\nonumber\\
	&\quad+\left(\frac{(n-2)(n-3)}{2}\left(1-\frac{2(1-a)}{n-1}\right)^k+(n-2)\left(\frac{an+a-2}{n-1}\right)^k\right),\;\text{ see Table \ref{Table}}\nonumber\\
	&\approx2+3(n-2)e^{-\left(\frac{1-a}{n-1}\right)k}+(n^2-5n+5)e^{-2\left(\frac{1-a}{n-1}\right)k}\\
	&\quad\quad\;+3\left(\frac{an-1}{n-1}\right)^k+(n-2)\left(a^k+\left(\frac{an+a-2}{n-1}\right)^k\right)\nonumber.
\end{align}
\begin{table}
	\begin{center}
		\begin{tabular}{c|c}
			\hline\\[-2ex]
			$\text{Partition of } n$ & $\text{ Eigenvalues of  }\widehat{\mathscr{P}}_a(R)\text{ corresponding to the partition of column }1$ \\\\[-2ex] \hline\hline \\[-1.5ex]
			$(n)$ & $1$ \text{ with algebraic multiplicity} $1$ \\\\[-2ex] \hline \\[-1.5ex]
			& $\frac{n+a-2}{n-1}$ \text{ with algebraic multiplicity} $n-2$ \\\\[-2ex]
			$(n-1,1)$ & \\
			& $\frac{an-1}{n-1}$ \text{ with algebraic multiplicity} $1$ \\\\[-2ex]  \hline \\[-1.5ex]
			&$\frac{n+2a-3}{n-1}$ \text{ with algebraic multiplicity} $\frac{(n-1)(n-4)}{2}$\\\\[-2ex]
			$(n-2,2)$ & \\ 
			& $a$ \text{ with algebraic multiplicity} $n-2$ \\\\[-2ex] \hline \\[-1.5ex]
			&$\frac{n+2a-3}{n-1}$ \text{ with algebraic multiplicity} $\frac{(n-2)(n-3)}{2}$\\\\[-2ex]
			$(n-2,1,1)$ & \\ 
			& $\frac{an+a-2}{n-1}$ \text{ with algebraic multiplicity} $n-2$ \\[1ex]\hline 
		\end{tabular}
	\end{center}
	\caption{Eigenvalues of $\widehat{\mathscr{P}}_a(R)\big|_{\lambda}$ for $\lambda=(n),(n-1,1),(n-2,2),\text{ and }(n-2,1,1)$.}\label{Table}
\end{table}
\begin{prop}\label{prop:LB_for_projected_S_n_case}
	For the random walk on $S_n$ driven by $\mathscr{P}_a$, we have the following:
	\[||\mathscr{P}_a^{*k}-U_{S_n}||_{\emph{TV}}\geq 1-\frac{4\left(E_{a,k}\left(\mathfrak{f}^2\right)-\left(E_{a,k}\left(\mathfrak{f}\right)\right)^2\right)}{\left(E_{a,k}(\mathfrak{f})\right)^2}-\frac{2}{E_{a,k}(\mathfrak{f})},\]
	where $E_{a,k}(\mathfrak{f})$ and $E_{a,k}(\mathfrak{f}^2)$ are given in \eqref{eq:E_k(.)} and \eqref{eq:E_k(.^2)} respectively.
\end{prop}
\begin{proof}
	If $\var_{a,k}(\mathfrak{f})$ denotes the variance of $\mathfrak{f}$ with respect to the probability measure $\mathscr{P}_a^{*k}$, then Chebychev's inequality implies that
	\begin{equation}\label{eq: LB CSV}
	\mathscr{P}_a^{*k}\left(\bigg\{\pi\in S_n: |\mathfrak{f}(\pi)-E_{a,k}(\mathfrak{f})|\leq \frac{E_{a,k}(\mathfrak{f})}{2}\bigg\}\right)\geq1-\frac{4\var_{a,k}(\mathfrak{f})}{\left(E_{a,k}(\mathfrak{f})\right)^2}.
	\end{equation}
	Again using $\mathfrak{f}\geq 0,\;E_{U}(\mathfrak{f})=1$ and the Markov's inequality, we have
	\begin{align}\label{eq: LB MRKV}
	U_{S_n}\left(\bigg\{\pi\in S_n: \mathfrak{f}(\pi)\geq \frac{E_{a,k}(\mathfrak{f})}{2}\bigg\}\right) & \leq  \frac{2E_{U}(\mathfrak{f})}{E_{a,k}(\mathfrak{f})}=\frac{2}{E_{a,k}(\mathfrak{f})}.
	\end{align}
	Now from the definition of total variation distance, we have
	\begin{align}\label{eq:LB}
	||\mathscr{P}_a^{*k}-U_{S_n}||_{\text{TV}}&=\sup_{A\subset S_n}|\mathscr{P}_a^{*k}(A)-U_{S_n}(A)|\nonumber\\
	&\geq \mathscr{P}_a^{*k}\left(\bigg\{\pi\in S_n: |\mathfrak{f}(\pi)-E_{a,k}(\mathfrak{f})|\leq \frac{E_{a,k}(\mathfrak{f})}{2}\bigg\}\right)\nonumber\\
	&\quad-U_{S_n}\left(\bigg\{\pi\in S_n: |\mathfrak{f}(\pi)-E_{a,k}(\mathfrak{f})|\leq \frac{E_{a,k}(\mathfrak{f})}{2}\bigg\}\right)\nonumber\\
	&\geq \mathscr{P}_a^{*k}\left(\bigg\{\pi\in S_n: |\mathfrak{f}(\pi)-E_{a,k}(\mathfrak{f})|\leq \frac{E_{a,k}(\mathfrak{f})}{2}\bigg\}\right)\nonumber\\
	&\quad-U_{S_n}\left(\bigg\{\pi\in S_n: \mathfrak{f}(\pi)\geq \frac{E_{a,k}(\mathfrak{f})}{2}\bigg\}\right).
	\end{align}
	Therefore the proposition follows from \eqref{eq:LB}, \eqref{eq: LB MRKV}, \eqref{eq: LB CSV}, and the definition of variance.
\end{proof}
We now come back to our main objective of this section i.e., computation of a lower bound of $||P^{*k}-U_{B_n}||_\text{{TV}}$. Let us define a homomorphism $f$ from $B_n$ onto $S_n$ as follows: For $\pi\in B_n$,
\begin{equation}\label{eq:definition_of_projection}
f: \pi\mapsto \left(f(\pi):i\mapsto|\pi(i)|,\text{ for }1\leq i\leq n\right).
\end{equation}
i.e., $f(\pi)\in S_n$ sends $i$ to $|\pi(i)|$ for $1\leq i\leq n$. Here $|\pi(i)|$ denotes the absolute value of $\pi(i)$. It can be checked that the mapping $f$ defined in \eqref{eq:definition_of_projection} is a homomorphism (this follows directly by considering $B_n$ as the \emph{wreath product} $S_2\wr S_n$). The surjectivity of $f$ follows from the definition. The homomorphism $f$ projects the flip-transpose top with random shuffle on $B_n$ to the transpose top with random shuffle on $S_n$ i.e., $Pf^{-1}=\mathscr{P}_{\frac{1}{n}}$. We now prove a lemma which will be useful in proving the main result of this section.
\begin{rem}\label{rem:transition_preservation_for_general_measure}
	Although we prove the upcoming lemma for the probability distribution $P$ on $B_n$, it is true if $P$ is replaced by any other probability distribution on $B_n$.
\end{rem}
\begin{lem}\label{lem:transition_preservation_of_the_projection}
	For any positive integer $k$ we have $\left(Pf^{-1}\right)^{*k}=P^{*k}f^{-1}$.
\end{lem}
\begin{proof}
	We use the first principle of mathematical induction on $k$. The base case for $k=1$ is true by definition. Now assume the induction hypothesis i.e., $\left(Pf^{-1}\right)^{*m}=P^{*m}f^{-1}$ for some positive integer $m>1$. Let $\pi\in S_n$ be chosen arbitrarily. Then for the inductive step $k=m+1$ we have the following:
	\begin{align}\label{eq:transition_preservation-1}
	\left(Pf^{-1}\right)^{*(m+1)}(\pi)&=\left(\left(Pf^{-1}\right)*\left(Pf^{-1}\right)^{*m}\right)(\pi)\nonumber\\
	&=\sum_{\{\xi,\zeta\in S_n:\;\xi\zeta=\pi\}}\left(Pf^{-1}\right)(\xi)\left(Pf^{-1}\right)^{*m}(\zeta)\nonumber\\
	&=\sum_{\{\xi,\zeta\in S_n:\;\xi\zeta=\pi\}}\left(Pf^{-1}\right)(\xi)\left(P^{*m}f^{-1}\right)(\zeta),\;\text{ by the induction hypothesis},\nonumber\\
	&=\sum_{\substack{\xi,\zeta\in S_n\\\xi\zeta=\pi}}P\left(f^{-1}(\xi)\right)P^{*m}\left(f^{-1}(\zeta)\right)\nonumber\\
	&=\sum_{\substack{\xi,\zeta\in S_n\\\xi\zeta=\pi}}\sum_{\substack{\xi^{\prime}\in f^{-1}(\xi)\\\zeta^{\prime}\in f^{-1}(\zeta)}}P(\xi^{\prime})P^{*m}(\zeta^{\prime}).
	\end{align}
	Now using the fact that $f$ is a homomorphism, we have the following:
	\begin{align*}
	\{(\xi^{\prime},\zeta^{\prime})\in f^{-1}(\xi)\times f^{-1}(\zeta):\;\xi,\zeta\in S_n\text{ and }\xi\zeta=\pi\}=\{(\xi^{\prime},\zeta^{\prime})\in B_n\times B_n:\;\xi^{\prime}\zeta^{\prime}\in f^{-1}(\pi)\}.
	\end{align*}
	Therefore the expression in \eqref{eq:transition_preservation-1} becomes
	\begin{align*}
	\sum_{\{\xi^{\prime},\zeta^{\prime}\in B_n:\;\xi^{\prime}\zeta^{\prime}\in f^{-1}(\pi)\}}P(\xi^{\prime})P^{*m}(\zeta^{\prime})&=\sum_{\pi^{\prime}\in f^{-1}(\pi)}\sum_{\substack{\xi^{\prime},\zeta^{\prime}\in B_n\\\xi^{\prime}\zeta^{\prime}=\pi^{\prime}}}P(\xi^{\prime})P^{*m}(\zeta^{\prime})\\
	&=\sum_{\pi^{\prime}\in f^{-1}(\pi)}P^{*(m+1)}(\pi^{\prime})=\left(P^{*(m+1)}f^{-1}\right)(\pi).
	\end{align*}
	Thus the lemma follows from the first principle of mathematical induction.
\end{proof}
\begin{proof}[Proof of Theorem \ref{thm:B_n lower Bound}]
	We know that, \emph{given two probability distributions $\mu$ and $\nu$ on $\Omega$ and a mapping $\psi:\Omega\rightarrow\Lambda$, we have $||\mu-\nu||_{\emph{TV}}\geq ||\mu\psi^{-1}-\nu\psi^{-1}||_{\emph{TV}}$, where $\Lambda$ is finite} \cite[Lemma 7.9]{LPW}. Therefore we have the following:
	\begin{align}\label{eq:LB-1}
	||P^{*k}-U_{B_n}||_\text{{TV}}&\geq||P^{*k}f^{-1}-U_{B_n}f^{-1}||_\text{{TV}}\nonumber\\
	&=||\left(Pf^{-1}\right)^{*k}-U_{S_n}||_\text{{TV}},\;\text{ by Lemma \eqref{lem:transition_preservation_of_the_projection} and }U_{B_n}f^{-1}=U_{S_n},\nonumber\\
	&=||\mathscr{P}_{\frac{1}{n}}^{*k}-U_{S_n}||_\text{{TV}},\;\text{ using }Pf^{-1}=\mathscr{P}_{\frac{1}{n}}.
	\end{align}
	Now setting $a=\frac{1}{n}$ in \eqref{eq:E_k(.)} and \eqref{eq:E_k(.^2)} we have
	\begin{align*}
	E_{\frac{1}{n},k}(\mathfrak{f})&\approx 1+(n-2)e^{-\frac{k}{n}}.\\
	E_{\frac{1}{n},k}(\mathfrak{f}^2)&\approx  2+3(n-2)e^{-\frac{k}{n}}+(n^2-5n+5)e^{-\frac{2k}{n}}+(n-2)\left(\frac{1+(-1)^k}{n^k}\right).
	\end{align*}
	Therefore Proposition \ref{prop:LB_for_projected_S_n_case} and \eqref{eq:LB-1} implies that
	\begin{align}\label{eq:LB-main}
	||P^{*k}-U_{B_n}||_\text{{TV}}\geq &1-\frac{2\left(3+3(n-2)e^{-\frac{k}{n}}-2(n-1)e^{-\frac{2k}{n}}+o(1)\right)}{\left(1+(n-2)e^{-\frac{k}{n}}\right)^2},\;\text{ for }k>1.
	\end{align}
	Now if $n$ is large, $c\ll0$ and $k=n\log n+cn$, then by \eqref{eq:LB-main}, we have the first part of this theorem.
	Again for any $\epsilon\in (0,1)$ and $k_n=\lfloor(1-\epsilon)n\log n\rfloor$ from \eqref{eq:LB-main}, we have
	\begin{equation}\label{eq:LB cutoff limit}
	1\geq ||P^{*k_n}-U_{B_n}||_{\text{TV}}\geq  1-\frac{2\left(3+3n^{\epsilon}+o(1)(n^{2\epsilon}+n^{\epsilon}+1)\right)}{\left(1+(1+o(1))n^{\epsilon}\right)^2},
	\end{equation}
	for large $n$. Therefore, the second part of this theorem follows from \eqref{eq:LB cutoff limit} and the fact that
	\[\lim_{n\rightarrow\infty}\frac{2\left(3+3n^{\epsilon}+o(1)(n^{2\epsilon}+n^{\epsilon}+1)\right)}{\left(1+(1+o(1))n^{\epsilon}\right)^2}=0.\]
\end{proof}
Therefore from the first part of Theorems \ref{thm:B_n Upper Bound} and \ref{thm:B_n lower Bound}, we can say that the mixing time for the 
flip-transpose top with random shuffle on $B_n$ is $O(n\log n)$ (i.e., order of $n\log n$). Furthermore, the second part of Theorems \ref{thm:B_n Upper Bound} and \ref{thm:B_n lower Bound} implies that this shuffle satisfies the cutoff phenomenon and the total variation cutoff for this shuffle occurs at $n\log n$.
\begin{rem}
	Let $0\leq\alpha\leq 1$. A generalisation of the flip-transpose top with random shuffle on $B_n$ can be considered, which we call \emph{the biased flip-transpose top with random shuffle on $B_n$}. Given an arrangement of $n$ distinct oriented cards in a row, choose a card uniformly at random and choose the last card. Then perform one of the following moves:
	\begin{enumerate}
		\item Transpose the chosen cards with probability $\frac{\alpha}{2}$.
		\item Transpose the chosen cards after flipping both the cards with probability $\frac{\alpha}{2}$.
		\item Transpose the chosen cards after flipping one of the cards with probability $\frac{1-\alpha}{2}$.
	\end{enumerate}
This is the random walk on $B_n$ driven by the probability measure $P_{\alpha}$ on $B_n$, defined below.
\begin{equation}\label{def of P_alpha}
P_{\alpha}(\pi)=
\begin{cases}
\frac{1}{n}\cdot\frac{\alpha}{2},&\text{if }\pi=(i,n)\text{ or }(-i,n)\text{ for }1\leq i\leq n,\;\text{ here }(n,n):=\1,\\
\frac{1}{n}\cdot\frac{1-\alpha}{2},&\text{if }\pi=(-n,n)(i,n)\text{ or }(-i,i)(i,n)\text{ for }1\leq i\leq n,\\
\;0,&\text{otherwise}.
\end{cases}
\end{equation}
The Fourier transform $\widehat{P}_{\alpha}(R)$ of $P_{\alpha}$ at the right regular representation $R$ is the transition matrix for this biased variant. Recall that the $n$th Young-Jucys-Murphy element of $B_n$ is $X_n$. Then $\widehat{P}_{\alpha}(R)$ is the action of $\frac{1}{2n}\left(\alpha\1+(1-\alpha)(-n,n)\right)\left(\1+(-n,n)+X_n\right)$ on $\mathbb{C}[B_n]$ by multiplication on the right. It can be easily seen that $\widehat{P}_{\alpha}(R)$ and $\widehat{P}(R)$ have the same set of eigenvectors when they act on the irreducible $B_n$-modules. Therefore using the arguments given in the proof of Theorem \ref{thm: evalues for B}, we can obtain the eigenvalues of $\widehat{P}_{\alpha}(R)$ as follows: \emph{For each $\mu=\left(\mu^{(1)},\mu^{(2)}\right)\in\yn$ satisfying $m:=|\mu^{(1)}|\in\{0,1,\dots,\lfloor\frac{n}{2}\rfloor\}$, let $T\in\tabD(n,\mu)$. Then $\frac{c(b_T(n))+1}{n}$ and $\frac{c(b_T(n))}{n}(2\alpha-1)$ are eigenvalues of $\widehat{P}_{\alpha}(R)$ with multiplicity $M(\mu)$ each}. Now using the fact $-1\leq2\alpha-1\leq 1$, we can conclude that $||P^{*k}-U_{B_n}||_{\text{TV}}$ and $||P_{\alpha}^{*k}-U_{B_n}||_{\text{TV}}$ have the same upper bound. Thus Theorem \ref{thm:B_n Upper Bound} is true if $P$ is replaced by $P_{\alpha}$. Moreover, the same mapping $f$ defined in \eqref{eq:definition_of_projection} projects the biased flip-transpose top with random shuffle on $B_n$ to the transpose top with random shuffle on $S_n$. Therefore $||P^{*k}-U_{B_n}||_{\text{TV}}$ and $||P_{\alpha}^{*k}-U_{B_n}||_{\text{TV}}$ have the same lower bound. Thus Theorem \ref{thm:B_n lower Bound} is true if $P$ is replaced by $P_{\alpha}$. \emph{Hence the biased flip-transpose top with random shuffle on $B_n$ satisfies total variation cutoff phenomenon with cutoff time $n\log n$}.
\end{rem}
\appendix
\section{Representation theory of demihyperoctahedral group $D_n$}
 In this section, we briefly discuss the irreducible representations of $D_n$ (detailed proofs are omitted). Our main aim is to look at the restriction of the irreducible representations of $B_n$ to $D_n$. 
 
Let us consider the one-dimensional character (or representation) $\xi:B_n\rightarrow(\{\pm 1\}, \cdot)$ of $B_n$. The action of $\xi$ on the generators of $B_n$ is defined by
\begin{equation}\label{def of xi}
\xi(\pi)=\begin{cases}
-1, &\text{if } \pi=(-1,1),\\
1, &\text{if } \pi=(i,i+1)\text{ for } 1\leq i \leq n-1.
\end{cases}
\end{equation}
It can be easily seen that $\ker(\xi)=D_n$ and the $B_n$-module $V\otimes\xi$ is irreducible if and only if the $B_n$-module $V$ is irreducible. We have already seen in Section \ref{sec:representation} that the irreducible representations of $B_n$ are indexed by $\yn$. If $\mu=(\mu^{(1)},\mu^{(2)})\in\yn$, then $\tilde{\mu}=(\mu^{(2)},\mu^{(1)})\in\yn$. Now from \cite[Proposition II.1.(ii)]{GK}, it follows that the irreducible $B_n$-modules $V^{\mu}\otimes\xi$ and $V^{\tilde{\mu}}$ are isomorphic for $\mu\in\yn$. 
\begin{thm}\label{thm: irraps restriction to D_n}
For the irreducible $B_n$-module $V^{\mu}$ indexed by $\mu=(\mu^{(1)},\mu^{(2)})\in\yn$, we have the following:
\begin{enumerate}
\item If $\mu^{(1)}\neq \mu^{(2)}$, then the restriction $V^{\mu}\downarrow^{B_n}_{D_n}$ of $V^{\mu}$ to $D_n$ is irreducible as a $D_n$-module. We denote this irreducible $D_n$-module by the same notation $V^{\mu}$. Moreover, if $\tilde{\mu}=(\mu^{(2)},\mu^{(1)})$, then $V^{\mu}$ and $V^{\tilde{\mu}}$ are isomorphic as $D_n$-modules. If $\nu\in\yn$ be such that $\nu\neq\mu$ and $\nu\neq\tilde{\mu}$, then $V^{\nu}$ and $V^{\mu}$ are non-isomorphic as $D_n$-modules.
\item If $\mu^{(1)}=\mu^{(2)}$, then the restriction $V^{\mu}\downarrow^{B_n}_{D_n}$ of $V^{\mu}$ to $D_n$ is a direct sum of two irreducible $D_n$-modules with the same dimension. We denote these irreducible $D_n$-modules by $V^{\mu}_{+}$ and $V^{\mu}_{-}$.
\end{enumerate}
\end{thm}
\begin{proof}
The proof follows by mimicking the steps of deducing the irreducible representations of $A_n$ from that of $S_n$ \cite[Theorem 4.4.2, Theorem 4.6.5]{Amri}. Here $A_n$ denotes the alternating group. For this proof, $B_n$ (respectively $D_n$) will play the role of $S_n$ (respectively $A_n$), and $\xi$ will play the role of the one-dimensional sign character of $S_n$.
\end{proof}
Let $\mathcal{S}$ be the collection of subsets $\Gamma$ of $\yn$ satisfying the following properties:
\begin{enumerate}
\item $\mu^{(1)}\neq \mu^{(2)}$ for each $(\mu^{(1)},\mu^{(2)})\in\Gamma$,
\item $(\mu^{(2)},\mu^{(1)})\notin\Gamma$ if and only if $(\mu^{(1)},\mu^{(2)})\in\Gamma$.
\end{enumerate}
Let $\Gamma_1$ be a maximal element of the poset $(\mathcal{S},\subseteq)$ and 
$\Gamma_2=\{(\mu^{(1)},\mu^{(2)})\in\yn: \mu^{(1)}=\mu^{(2)}\}$. Then from Theorem \ref{thm: irraps restriction to D_n} and the observation
\begin{align*}
&\quad\quad \displaystyle\sum_{\mu\in\Gamma_1}\left(\text{dim}(V^{\mu})\right)^2+\displaystyle\sum_{\mu\in\Gamma_2}\left(\left(\text{dim}(V^{\mu}_{+})\right)^2+\left(\text{dim}(V^{\mu}_{-})\right)^2\right)\\
&=\frac{1}{2}\left(2\displaystyle\sum_{\mu\in\Gamma_1}\left(\text{dim}(V^{\mu})\right)^2+\displaystyle\sum_{\mu\in\Gamma_2}\left(\text{dim}(V^{\mu})\right)^2\right)=\frac{|B_n|}{2}=|D_n|,
\end{align*}
all the irreducible $D_n$-modules are given by $\{V^{\mu}:\mu\in\Gamma_1\}\cup\{V^{\mu}_{+},V^{\mu}_{-}:\mu\in\Gamma_2\}$. 
\section{A random walk on $D_n$ analogous to the walk on $B_n$ driven by $P$}
Let us consider the random walk on the demihyperoctahedral group $D_n$ driven by the probability measure $Q$ on $D_n$ defined as follows:
\begin{equation}\label{def of Q}
Q(\pi)=
\begin{cases}
\frac{1}{2n-1},&\text{if }\pi=\1,\text{ the identity element of }D_n,\\
\frac{1}{2n-1},&\text{if }\pi=(i,n)\text{ for }1\leq i\leq n-1,\\
\frac{1}{2n-1},&\text{if }\pi=(-i,n)\text{ for }1\leq i\leq n-1,\\
0,&\text{otherwise}.
\end{cases}
\end{equation}
It can be easily seen that the support of $Q$ generates $D_n$ and hence this random walk is irreducible. Moreover, this random walk is aperiodic too. Thus the distribution after $k\text{th}$ transition for this random walks will converge to $U_{D_n}$ as $k\rightarrow\infty$. Let us recall that $\widehat{Q}(R)$ is the Fourier transform of $Q$ at the right regular representation $R$ of $D_n$. The transition matrix for the random walk on $D_n$ driven by $Q$ is the transpose of $\widehat{Q}(R)$. To find the eigenvalues of $\widehat{Q}(R)$ we will use the representation theory of $D_n$.
\begin{thm}\label{thm: evalues for D}
The eigenvalues of $\widehat{Q}(R)$ are given by
\begin{enumerate}
\item If $\mu=\left(\mu^{(1)},\mu^{(2)}\right)\in\Gamma_1$, then for each $T\in\tabD(n,\mu),\;\frac{2c(b_T(n))+1}{2n-1}$ is an eigenvalue of $\widehat{Q}(R)$ with multiplicity $\emph{dim}(V^{\mu})$.
\item If $\mu=\left(\mu^{(1)},\mu^{(2)}\right)\in\Gamma_2$, then for each $T\in\tabD(n,\mu),\;\frac{2c(b_T(n))+1}{2n-1}$ is an eigenvalue of $\widehat{Q}(R)$ with multiplicity $\frac{1}{2}\emph{dim}(V^{\mu})$.
\end{enumerate} 
Recall $c(b_T(n))$ is the content of the box containing $n$ in $T$. 
\end{thm}
\begin{proof}
We have $\widehat{Q}(R)=\frac{1}{2n-1}\left(X_n+\1\right)$, where $X_n$ is the $n\text{th}$ Young-Jucys-Murphy element of $B_n$ and $\1$ is the identity element of $D_n$. Here we identify the elements of $D_n$($\subseteq B_n$) by the elements of $B_n$.

For $\mu=\left(\mu^{(1)},\mu^{(2)}\right)\in\Gamma_1$, we have $\mu^{(1)}\neq\mu^{(2)}$. Therefore the restriction of irreducible $B_n$-module $V^{\mu}$ to $D_n$ is irreducible (Theorem \ref{thm: irraps restriction to D_n}). Now for each $T\in\tabD(n,\mu)$, let $v_T$ be the Gelfand-Tsetlin vector of $V^{\mu}$ satisfying $X_nv_T=2c(b_T(n))v_T$. Also, we know that $\{v_T:T\in\tabD(n,\mu)\}$ forms a basis of $V^{\mu}$. Therefore the eigenvalues of $\widehat{Q}(R)$ on the irreducible $D_n$-module $V^{\mu}$ are given by $\frac{2c(b_T(n))+1}{2n-1}$ for each $T\in\tabD(n,\mu)$. Since the multiplicity of every irreducible representation in the right regular representation is equal to its dimension, therefore the multiplicity of these eigenvalues are dim$(V^{\mu})$.

Now for $\mu=\left(\mu^{(1)},\mu^{(2)}\right)\in\Gamma_2$ we  have $\mu^{(1)}=\mu^{(2)}$. Then the restriction of the irreducible $B_n$-module $V^{\mu}$ to $D_n$ splits into two irreducible $D_n$-modules $V^{\mu}_{+}$ and $V^{\mu}_{-}$ (Theorem \ref{thm: irraps restriction to D_n}). In this case also $v_T$ is the Gelfand-Tsetlin vector of $V^{\mu}$ and $\{v_T:T\in\tabD(n,\mu)\}$ forms a basis of $V^{\mu}_{+}\oplus V^{\mu}_{-}$. Therefore, by similar arguments in case of $\mu^{(1)}\neq\mu^{(2)}$, the eigenvalues of $\widehat{Q}(R)$ on the irreducible $D_n$-modules $V^{\mu}_{+}$ and $V^{\mu}_{-}$ are given by $\frac{2c(b_T(n))+1}{2n-1}$ for each $T\in\tabD(n,\mu)$. The multiplicity of these eigenvalues are $\frac{1}{2}\text{dim}(V^{\mu})$ ( $\because$ dim$(V^{\mu}_{+})=\text{dim}(V^{\mu}_{-})=\frac{1}{2}\text{dim}(V^{\mu})$).
\end{proof}
\begin{thm}\label{thm:D_n Upper Bound}
For the random walk on $D_n$ driven by $Q$, we have the following:
\begin{enumerate}
\item $||Q^{*k}-U_{D_n}||_{\text{TV}}<\sqrt{e+1}\;e^{-c}+o(1)$, for $k\geq\left(n-\frac{1}{2}\right)(\log n+c)$ and $c>0$.
\item $\lim\limits_{n\rightarrow\infty}||Q^{*k_n}-U_{D_n}||_{\text{TV}}=0$, for any $\epsilon\in (0,1)$ and $k_n=\lfloor(1+\epsilon)\left(n-\frac{1}{2}\right)\log n\rfloor.$
\end{enumerate}
\end{thm}
\begin{proof}
Using Lemma \ref{Upper Bound Lemma} and following similar steps of Theorem \ref{thm:B_n Upper Bound}, we have 
\begin{equation}\label{eq: D_n UB}
4||Q^{*k}-U_{D_n}||^2_{\text{TV}}\leq 2(1+e)\left(e^{n^2e^{-\frac{4k}{2n-1}}}-1\right)+e^{-\frac{4k}{2n-1}},\quad\text{for }k\geq \left(n-\frac{1}{2}\right)\log n.
\end{equation}
Now if $k\geq \left(n-\frac{1}{2}\right)(\log n+c)$ and $c>0$, then the right hand side of \eqref{eq: D_n UB} becomes
\[2(e+1)\left(e^{e^{-2c}}-1\right)+\frac{e^{-2c}}{n^2}<(4e+4)e^{-2c}+o(1).\]
This proves the first part of the theorem. Now for $\epsilon\in(0,1),\;k_n=\lfloor(1+\epsilon)\left(n-\frac{1}{2}\right)\log n\rfloor$ implies, $k_n\geq(1+\epsilon)\left(n-\frac{1}{2}\right)\log n$. Thus the right hand side of \eqref{eq: D_n UB} is bounded above by $2(e+1)\left(e^{\frac{1}{n^{2\epsilon}}}-1\right)+\frac{1}{n^{2(1+\epsilon)}}$. Therefore the proof of the second part follows from
\[\lim\limits_{n\rightarrow\infty}2(e+1)\left(e^{\frac{1}{n^{2\epsilon}}}-1\right)+\frac{1}{n^{2(1+\epsilon)}}=0.\qedhere\]
\end{proof}
Now we obtain a lower bound for the total variation distance $||Q^{*k}-U_{D_n}||_{\text{TV}}$. Recall the homomorphism $f$ defined in \eqref{eq:definition_of_projection} and set $f^{\prime}=f\big|_{D_n}$, the restriction of $f$ to $D_n$. Then $f^{\prime}$ projects the random walk on $D_n$ driven by $Q$ to the random walk on $S_n$ driven by $\mathscr{P}_{\frac{1}{2n-1}}$. Thus we have $U_{D_n}f^{\prime-1}=U_{S_n}$ and $Qf^{\prime-1}=\mathscr{P}_{\frac{1}{2n-1}}$. Now using the arguments used in the proof of Lemma \ref{lem:transition_preservation_of_the_projection}, we can conclude that \[\mathscr{P}_{\frac{1}{2n-1}}^{*k}=\left(Qf^{\prime-1}\right)^{*k}=Q^{*k}f^{\prime-1}.\]
Therefore using \cite[Lemma 7.9]{LPW}, we have the following:
\begin{equation}\label{eq:LB-1_for_D}
||Q^{*k}-U_{D_n}||_{\text{TV}}\geq||Q^{*k}f^{\prime-1}-U_{D_n}f^{\prime-1}||_{\text{TV}}=\left|\left|\mathscr{P}_{\frac{1}{2n-1}}^{*k}-U_{S_n}\right|\right|_{\text{TV}}.
\end{equation}
\begin{thm}\label{thm:D_n lower Bound}
	For the random walk on $D_n$ driven by $Q$, we have the following:
	\begin{enumerate}
		\item For large $n,\;||Q^{*k}-U_{D_n}||_{\emph{TV}}\geq 1-\frac{2\left(3+3e^{-c}+o(1)(e^{-2c}+e^{-c}+1)\right)}{\left(1+(1+o(1))e^{-c}+o(1)\right)^2}$, when $k=(n-\frac{1}{2})(\log n+c)$ and $c\ll0$.
		\item $\lim\limits_{n\rightarrow\infty}||Q^{*k_n}-U_{D_n}||_{\emph{TV}}=1$, for any $\epsilon\in (0,1)$ and $k_n=\lfloor(1-\epsilon)\left(n-\frac{1}{2}\right)\log n\rfloor.$
	\end{enumerate}
\end{thm}
\begin{proof}
	 Setting $a=\frac{1}{2n-1}$ in \eqref{eq:E_k(.)} and \eqref{eq:E_k(.^2)} we have 
	\begin{align*}
	E_{\frac{1}{2n-1},k}\left(\mathfrak{f}\right)&\approx 1+(n-2)e^{-\frac{2k}{2n-1}}+\left(-\frac{1}{2n-1}\right)^k.\\
	E_{\frac{1}{2n-1},k}\left(\mathfrak{f}^2\right)&\approx 2+3(n-2)e^{-\frac{2k}{2n-1}}+(n^2-5n+5)e^{-\frac{4k}{2n-1}}\\
	&\quad\quad+3\left(-\frac{1}{2n-1}\right)^k+(n-2)\left(\left(\frac{1}{2n-1}\right)^k+\left(\frac{-3}{2n-1}\right)^k\right).
	\end{align*}
	Therefore Proposition \ref{prop:LB_for_projected_S_n_case} and \eqref{eq:LB-1_for_D} implies that
	\begin{equation}\label{eq:LB-main_for_D}
	||Q^{*k}-U_{D_n}||_{\text{TV}}\geq 1-\frac{2\left(3+3(n-2)e^{-\frac{2k}{2n-1}}-2(n-1)e^{-\frac{4k}{2n-1}}+o(1)\right)}{\left(1+(n-2)e^{-\frac{2k}{2n-1}}+o(1)\right)^2},\;\text{ for }k>1.
	\end{equation}
	Now if $n$ is large, $c\ll0$ and $k=(n-\frac{1}{2})(\log n+c)$, then by \eqref{eq:LB-main_for_D}, we have the first part of this theorem. Again for any $\epsilon\in (0,1)$ and $k_n=\lfloor(1-\epsilon)(n-\frac{1}{2})\log n\rfloor$ from \eqref{eq:LB-main_for_D}, we have
	\begin{equation}\label{eq:LB cutoff limit for D_n}
	1\geq ||Q^{*k_n}-U_{D_n}||_{\text{TV}}\geq  1-\frac{2\left(3+3n^{\epsilon}+o(1)(n^{2\epsilon}+n^{\epsilon}+1)\right)}{\left(1+(1+o(1))n^{\epsilon}+o(1)\right)^2}
	\end{equation}
	for large $n$. Therefore, the second part of this theorem follows from \eqref{eq:LB cutoff limit for D_n} and the fact that
	\[\lim_{n\rightarrow\infty}\frac{2\left(3+3n^{\epsilon}+o(1)(n^{2\epsilon}+n^{\epsilon}+1)\right)}{\left(1+(1+o(1))n^{\epsilon}+o(1)\right)^2}=0.\]
\end{proof}
Therefore from the first part of Theorems \ref{thm:D_n Upper Bound} and \ref{thm:D_n lower Bound}, we can say that the mixing time for the random walk on $D_n$ driven by $Q$ is $O\left(\left(n-\frac{1}{2}\right)\log n\right)$. Furthermore, the second part of Theorems \ref{thm:D_n Upper Bound} and \ref{thm:D_n lower Bound} implies that this shuffle satisfies the cutoff phenomenon and the total variation cutoff for this shuffle occurs at $\left(n-\frac{1}{2}\right)\log n$.

\bibliography{ref}{}
\bibliographystyle{plain}
\end{document}